\documentclass{article}
\usepackage[latin1]{inputenc}
\usepackage[T1]{fontenc}
\usepackage{amsmath,amssymb}
\usepackage{amsfonts,amsthm}
\usepackage{url}
\usepackage{geometry}
\usepackage[hypertexnames=false]{hyperref}
\usepackage{hypcap}
\usepackage{color}
\usepackage{graphicx}
\usepackage{caption}
\usepackage{subcaption}
\usepackage{multirow}
\usepackage{colortbl}
\usepackage{epstopdf}

\newcommand{\C}{\mathbb{C}}
\newcommand{\D}{\mathbb{D}}

\newcommand{\N}{\mathbb{N}}

\newcommand{\Z}{\mathbb{Z}}

\newcommand{\Fc}{\mathcal{F}}

\newcommand{\Lc}{\mathcal{L}}

\newcommand{\from}{\colon}

\DeclareMathOperator{\clos}{clos}

\DeclareMathOperator{\spec}{sp}

\renewcommand{\epsilon}{\varepsilon}
\renewcommand{\phi}{\varphi}
\renewcommand{\theta}{\vartheta}

\providecommand{\abs}[1]{\left\lvert#1\right\rvert}
\providecommand{\norm}[1]{\left\lVert#1\right\rVert}
\providecommand{\set}[1]{\left\{ #1\right\}}

\newtheorem{thm}{Theorem}

\newtheorem{lem}[thm]{Lemma}
\newtheorem{prop}[thm]{Proposition}
\newtheorem{cor}[thm]{Corollary}

\theoremstyle{definition}
\newtheorem{defn}[thm]{Definition}
\newtheorem{rem}[thm]{Remark}

\begin{document}

\title{Symmetries of the Feinberg-Zee Random Hopping Matrix}
\author{Raffael Hagger\footnote{raffael.hagger@tuhh.de}}
\maketitle

\begin{abstract}
We study the symmetries of the spectrum of the Feinberg-Zee Random Hopping Matrix introduced in \cite{FeZee} and studied in various papers therafter (e.g. \cite{ChaChoLi}, \cite{ChaChoLi2}, \cite{ChaDa}, \cite{Ha}, \cite{HoOrZee}). In \cite{ChaDa}, Chandler-Wilde and Davies proved that the spectrum of the Feinberg-Zee Random Hopping Matrix is invariant under taking square roots, which implied that the unit disk is contained in the spectrum (a result already obtained slightly earlier in \cite{ChaChoLi}). In a similar approach we show that there is an infinite sequence of symmetries at least in the periodic part of the spectrum (which is conjectured to be dense). Using these symmetries and the result of \cite{ChaDa}, we can exploit a considerably larger part of the spectrum than the unit disk. As a further consequence we find an infinite sequence of Julia sets contained in the spectrum. These facts may serve as a part of an explanation of the seemingly fractal-like behaviour of the boundary.
\end{abstract}



\section{Introduction}

In recent years some progress has been made in the study of non-self-adjoint random operators (see \cite{ChaChoLi}, \cite{ChaChoLi2}, \cite{ChaDa}, \cite{Ha}, \cite{HoOrZee} and references therein). However, still a lot of questions remain open even in the tridiagonal case. In particular, the spectrum of most tridiagonal random operators is yet unknown. Although the techniques used here are possibly also suited for more general types of operators, we focus on the Feinberg-Zee Random Hopping Matrix \cite{FeZee} here. It is defined as follows:
\[A^b := \begin{pmatrix} \ddots & \ddots & & & \\ \ddots & 0 & 1 & & \\ & b_0 & 0 & 1 & \\ & & b_1 & 0 & \smash{\ddots} \\ & & & \ddots & \ddots \end{pmatrix} \in \Lc(\ell^2(\Z))\]
for a random sequence $b \in \set{\pm 1}^{\Z}$, i.e. a sequence with randomly (i.i.d.) distributed entries. Similarly we define $A^c$ for arbitrary sequences $c \in \set{\pm 1}^{\Z}$. The spectrum
\[\spec(A^b) := \set{\lambda \in \C : A^b - \lambda I \text{ is not invertible}}\]
of $A^b$ is independent of the sequence $b$ in the following sense:
\begin{prop}
(e.g. \cite[Lemma 2.3, Theorem 2.5]{ChaChoLi2})
For a sequence of i.i.d. random variables $(b_j)_{j \in \Z}$ taking values in $\set{\pm 1}$ with non-zero probability each, the spectrum of $A^b$ is given by
\begin{equation} \label{eq1}
\spec(A^b) = \Sigma := \bigcup\limits_{c \in \set{\pm 1}^{\Z}} \spec(A^c)
\end{equation}
almost surely.
\end{prop}

In \cite{ChaChoLi} Chandler-Wilde, Chonchaiya and Lindner discovered a beautiful connection to the Sierpinski triangle. This connection was then used to show that the unit disk is contained in $\Sigma$, which disproved earlier conjectures that $\Sigma$ may be of fractal dimension. In \cite{ChaChoLi2} the same authors also gave an upper bound using the numerical range. This upper bound was further improved in \cite{Ha}. An image of $\Sigma$ (or more precisely: what it is conjectured to be) and the upper bounds computed in \cite{ChaChoLi2} and \cite{Ha} are provided in Figure 1.

\begin{figure}[h!]
\begin{center}
\includegraphics[scale = 1]{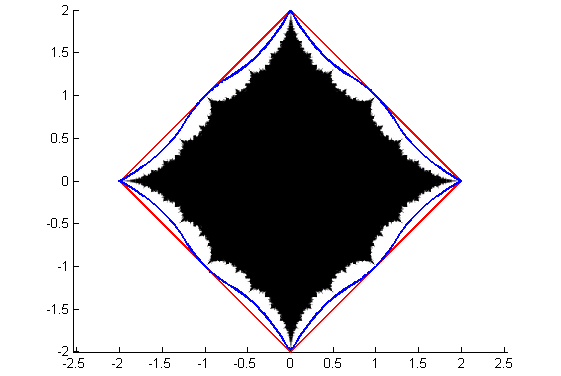}
\end{center}

\caption{The conjectured shape of $\Sigma$ (black), the upper bound computed in \cite{ChaChoLi2} (red) and the upper bound computed in \cite{Ha} (blue).}
\end{figure}

Another way of deriving lower bounds is using equation \eqref{eq1} directly. The spectra of many operators on the right hand side can be computed explicitly. This is in particular the case for periodic operators. We call an operator $A \in \Lc(\ell^2(\Z))$ (not necessarily tridiagonal) $m$-periodic if $A_{i,j} = A_{i+m,j+m}$ for all $i,j \in \Z$. In the tridiagonal case, we will also use the notation $A^k_{per}$ for $k \in \set{\pm 1}^m$, i.e.
\[A^k_{per} = \begin{pmatrix} \ddots & \ddots & & & & & & \\ \ddots & 0 & 1 & & & & & \\ & k_m & 0 & 1 & & & & \\ & & k_1 & \ddots & \ddots & & & \\ & & & \ddots & \ddots & 1 & & \\ & & & & k_m & 0 & 1 & \\ & & & & & k_1 & 0 & \ddots \\ & & & & & & \ddots & \ddots \end{pmatrix} \in \Lc(\ell^2(\Z)).\]
Note that $k$ is not unique because $(k_1, \ldots, k_m) \in \set{\pm 1}^m$ and $(k_1, \ldots, k_m,k_1, \ldots, k_m) \in \set{\pm 1}^{2m}$ define the same operator. We will make use of this fact later on.

The good thing about periodic operators is that they can be diagonalized in some sense.

\begin{prop} \label{thmD}
(e.g. \cite[Theorem 4.4.9]{Davies2})\\
Let $m \in \N$, let $A \in \Lc(\ell^2(\Z))$ be an $m$-periodic operator and define
\[a_r := \begin{pmatrix} A_{rm+1,1} & \hdots & A_{rm+1,m} \\ \vdots & \ddots & \vdots \\ A_{rm+m,1} & \hdots & A_{rm+m,m} \end{pmatrix}\]
for all $r \in \Z$. If $\sum\limits_{r \in \Z} \norm{a_r} < \infty$, then $A$ is unitarily equivalent to the (generalized) multiplication operator $M_a \in \Lc(L^2([0,2\pi),\C^m))$. The function $a \in L^{\infty}([0,2\pi),\C^{m \times m})$ is called the symbol of $A$ and given by
\[a(\phi) = \sum\limits_{r \in \Z} a_r e^{ir\phi} \quad (\phi \in [0,2\pi)).\]
Moreover, $\spec(A) = \set{\lambda \in \C : \det(a(\phi) - \lambda I_m) = 0 \text{ for some } \phi \in [0,2\pi)} = \bigcup\limits_{\phi \in [0,2\pi)} \spec(a(\phi))$.
\end{prop}

In the following, we will always have a finite number of non-vanishing coefficients $a_r$. Thus we will not have to worry about the convergence of $\sum\limits_{r \in \Z} \norm{a_r}$. Proposition \ref{thmD} enables us to compute the spectra of periodic operators explicitly (numerically). We call the set
\[\pi_\infty := \bigcup\limits_{m \in \N} \bigcup\limits_{k \in \set{\pm 1}^m} \spec(A^k_{per})\]
the periodic part of $\Sigma$. Another related set is the set of eigenvalues of finite matrices of this kind. For $n \in \N$ and $k \in \set{\pm 1}^n$ we define
\[A^k_{fin} := \begin{pmatrix} 0 & 1 & & \\ k_1 & \ddots & \ddots & \\ & \ddots & \ddots & 1 \\ & & k_n & 0 \end{pmatrix} \in \C^{(n+1) \times (n+1)}.\]
Then we call the set
\[\sigma_\infty := \bigcup\limits_{n \in \N} \bigcup\limits_{k \in \set{\pm 1}^n} \spec(A^k_{fin})\]
the finite part of $\Sigma$. It is clear by equation \eqref{eq1} that $\pi_\infty$ is a subset of $\Sigma$. Furthermore, it was shown in \cite{ChaChoLi2} that $\sigma_\infty \subset \pi_\infty$ holds. It was then conjectured in \cite{ChaChoLi} that $\clos(\sigma_\infty) = \clos(\pi_\infty) = \Sigma$ and that $\Sigma$ is a simply connected set which is the closure of its interior and which has a fractal boundary. That the first equality holds was recently shown in \cite{Ha2} whereas the second equality and the other assertions remain open. Some light was shed on this question in \cite{ChaDa}, where it was proved that $\D \subset \clos(\pi_\infty)$. Combined with the result of \cite{Ha2} and equation \eqref{eq1}, this implies
\begin{equation} \label{eq2}
\D \subset \clos(\sigma_\infty) = \clos(\pi_\infty) \subset \Sigma.
\end{equation}
The result of Chandler-Wilde and Davies in \cite{ChaDa} is based on the observation that $\pi_\infty$ (and also $\Sigma$) is invariant under taking square roots, i.e.
\[\lambda^2 \in \pi_\infty \Longrightarrow \lambda \in \pi_\infty.\]
In this paper we extend this result to an infinite number of symmetries. We prove that for $p \in S$, where $S$ is an infinite set of polynomials made precise below, the following holds:
\[p(\lambda) \in \pi_\infty \Longrightarrow \lambda \in \pi_\infty.\]
In other words, $\pi_\infty$ (and hence also $\clos(\pi_\infty)$) is invariant under taking roots of polynomials $p \in S$. This implies that \eqref{eq2} can be extended to
\[p^{-1}(\D) \subset p^{-1}(\clos(\sigma_\infty)) = p^{-1}(\clos(\pi_\infty)) \subset \clos(\pi_\infty) \subset \Sigma\]
for all $p \in S$. Thus
\begin{equation} \label{eq2.5}
\bigcup\limits_{p \in S} p^{-1}(\D) \subset \clos(\sigma_\infty) = \clos(\pi_\infty) \subset \Sigma.
\end{equation}
This improvement in comparison with \eqref{eq2} is significant as Figure 2 shows.

\begin{figure}[ht!]
\begin{center}
\includegraphics[width = .75\textwidth, trim = 15mm 5mm 10mm 5mm, clip = true]{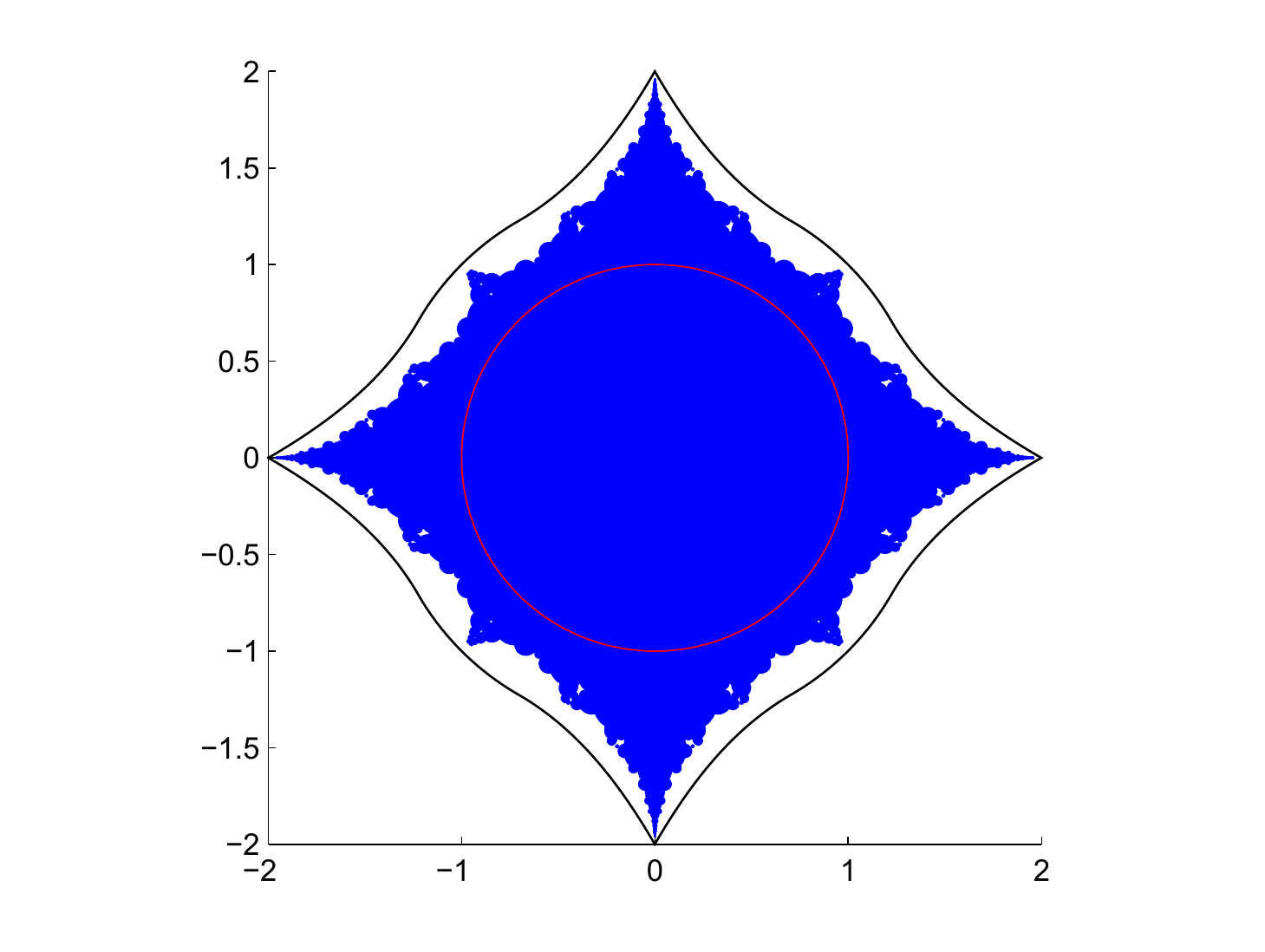}
\caption{The left-hand side of \eqref{eq2.5} (polynomials in $S$ up to order 15 to be precise), the upper bound to $\Sigma$ computed in \cite{Ha} and the unit circle as a reference.}
\end{center}
\end{figure}

Clearly, this construction can also be iterated, i.e.
\[p^{-n}(\D) \subset p^{-n}(\clos(\sigma_\infty)) = p^{-n}(\clos(\pi_\infty)) \subset \clos(\pi_\infty) \subset \Sigma\]
and hence
\[\bigcup\limits_{n \in \N} p^{-n}(\D) \subset \bigcup\limits_{n \in \N} p^{-n}(\clos(\sigma_\infty)) = \bigcup\limits_{n \in \N} p^{-n}(\clos(\pi_\infty)) \subset \clos(\pi_\infty) \subset \Sigma.\]
This implies that $\Sigma$ contains an infinite sequence of (presumably filled) Julia sets (see Remark \ref{remark} below), e.g. the set indicated in Figure 3.

\begin{figure}[ht!]
\begin{center}
\includegraphics[width = .75\textwidth,trim = 15mm 5mm 10mm 5mm, clip = true]{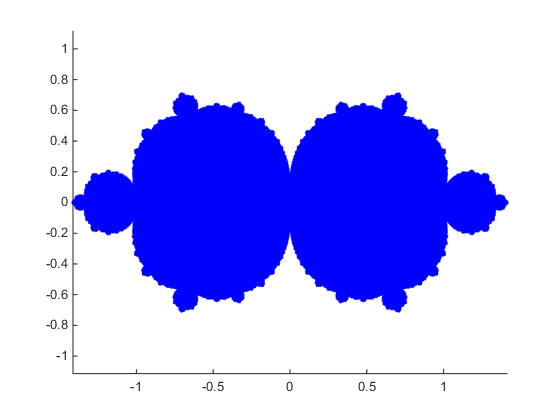}
\caption{The filled Julia set corresponding to $p(\lambda) = \lambda^3 - \lambda$.}
\end{center}
\end{figure}

These two approaches can also be combined as follows. Let $T$ be the closure of $S$ with respect to composition, i.e.
\[T = \set{q \from \C \to \C : q = p_1 \circ \ldots \circ p_n \text{ for } p_1, \ldots, p_n \in S, n \in \N}.\]
Then
\[\bigcup\limits_{q \in T} q^{-1}(\D) \subset \clos(\sigma_\infty) = \clos(\pi_\infty) \subset \Sigma.\]
In this way one can construct even more Julia sets that are contained in $\Sigma$. This richness of symmetries might be a part of an explanation of the seemingly fractal boundary of $\Sigma$. Surely, this observation needs further investigation.

Note that in \cite{ChaDa} it was also shown that
\[\lambda \in \Sigma \Longrightarrow \pm\sqrt{\lambda} \in \Sigma.\]
Here we only have the (possibly weaker) statement
\[\lambda \in \Sigma \Longrightarrow p^{-1}(\set{\lambda}) \cap \Sigma \neq \emptyset\]
for $p \in S$.

In addition to the polynomial symmetries mentioned above, there are also the following symmetries (see \cite[Lemma 3.4]{ChaChoLi2}):
\[\lambda \in \pi_\infty \Longrightarrow i\lambda,\bar{\lambda} \in \pi_\infty \quad \text{and} \quad \lambda \in \Sigma \Longrightarrow i\lambda,\bar{\lambda} \in \Sigma.\]

We start with some well-known preparatory results and end up with the two main theorems of this paper. A short list of polynomials in $S$ and selected pictures of new subsets of $\Sigma$ are provided at the end.

\section{Symmetries}

Let $m \in \N$ and $k \in \set{\pm 1}^m$. Then we denote the corresponding $m$-periodic operator
\[\begin{pmatrix} \ddots & \ddots & & & & & & \\ \ddots & 0 & 1 & & & & & \\ & k_m & 0 & 1 & & & & \\ & & k_1 & \ddots & \ddots & & & \\ & & & \ddots & \ddots & 1 & & \\ & & & & k_m & 0 & 1 & \\ & & & & & k_1 & 0 & \ddots \\ & & & & & & \ddots & \ddots \end{pmatrix} \in \Lc(\ell^2(\Z))\]
by $A^k_{per}$. By Proposition \ref{thmD}, we can use the symbol $a^k$ to compute the spectrum of $A^k_{per}$. In our case, due to tridiagonality, the formula can be simplified as follows.

\begin{lem} \label{lemma1}
Let $m \in \N$, $k \in \set{\pm 1}^m$ and let $a^k$ denote the symbol of $A^k_{per}$. Then the only $\phi$-dependent term in the characteristic polynomial of $a^k(\phi)$ is the term of order zero. More precisely, there exists a polynomial $p_k \from \C \to \C$ of degree $m$ such that
\begin{equation} \label{eq3}
\det(a^k(\phi) - \lambda I_m) = (-1)^m\left(p_k(\lambda) - e^{i\phi}\prod\limits_{j = 1}^m k_j - e^{-i\phi}\right)
\end{equation}
for all $\phi \in [0,2\pi)$. The polynomial $p_k$ is monic and given by
\[p_k(\lambda) = (-1)^m\left(\det \begin{pmatrix} -\lambda & 1 & & \\ k_1 & \ddots & \ddots & \\ & \ddots & \ddots & 1 \\ & & k_{m-1} & -\lambda \end{pmatrix} - k_m\det \begin{pmatrix} -\lambda & 1 & & \\ k_2 & \ddots & \ddots & \\ & \ddots & \ddots & 1 \\ & & k_{m-2} & -\lambda \end{pmatrix}\right).\]
Furthermore, $p_k$ is an even (odd) function if $m$ is even (odd).
\end{lem}

\begin{proof}
The symbol $a^k$ is given by
\[a^k(\phi) = \begin{pmatrix} 0 & 1 & & k_me^{i\phi} \\ k_1 & \ddots & \ddots & \\ & \ddots & \ddots & 1 \\e^{-i\phi} & & k_{m-1} & 0 \end{pmatrix}.\]
Using Laplace's formula, we get
\begin{align*}
\det(a^k(\phi) - \lambda I_m) &= -\lambda \det \begin{pmatrix} -\lambda & 1 & & \\ k_2 & \ddots & \ddots & \\ & \ddots & \ddots & 1 \\ & & k_{m-1} & -\lambda \end{pmatrix}\\
& \quad \, - k_1\left(\det \begin{pmatrix} -\lambda & 1 & & \\ k_3 & \ddots & \ddots & \\ & \ddots & \ddots & 1 \\ & & k_{m-1} & -\lambda \end{pmatrix} + (-1)^{m-2}k_me^{i\phi} \prod\limits_{j = 2}^{m-1} k_j\right)\\
& \quad \, + (-1)^{m-1}e^{-i\phi}\left(1 + (-1)^{m-2}k_me^{i\phi} \det \begin{pmatrix} -\lambda & 1 & & \\ k_2 & \ddots & \ddots & \\ & \ddots & \ddots & 1 \\ & & k_{m-2} & -\lambda \end{pmatrix}\right)\\
&= \det \begin{pmatrix} -\lambda & 1 & & \\ k_1 & \ddots & \ddots & \\ & \ddots & \ddots & 1 \\ & & k_{m-1} & -\lambda \end{pmatrix} - k_m\det \begin{pmatrix} -\lambda & 1 & & \\ k_2 & \ddots & \ddots & \\ & \ddots & \ddots & 1 \\ & & k_{m-2} & -\lambda \end{pmatrix}\\
& \quad \, + (-1)^{m-1}e^{i\phi}\prod\limits_{j = 1}^m k_j+ (-1)^{m-1}e^{-i\phi}.
\end{align*}
That $p_k$ is an even (odd) function if $m$ is even (odd) follows easily by induction over $m$. 
\end{proof}

The most important part of Lemma \ref{lemma1} is that there are no mixed terms of $\lambda$ and $\phi$ in equation \eqref{eq3}. This leads to the fact (see Corollary \ref{cor1} below) that the spectrum of every periodic operator $A^k_{per}$ can be written as the preimage of the interval $[-2,2]$ under some polynomial $p_k$. In this way the various parts of $\pi_\infty$ are connected. We will make great use of this fact in Theorem \ref{thm2}. But first observe the following. The term in \eqref{eq3} involving $\phi$ can be simplified to $-2\cos(\phi)$ or $2i\sin(\phi)$ depending on the product $\prod\limits_{j = 1}^m k_j$. To avoid unnecessary paperwork, we give the following definition.

\begin{defn} \label{defn1}
Let $m \in \N$ and $k \in \set{\pm 1}^m$. Then we call $k$ even if $\prod\limits_{j = 1}^m k_j = 1$ and odd if $\prod\limits_{j = 1}^m k_j = -1$.
\end{defn}

Note that we can always assume that a periodic operator $A^k_{per}$ has an even period $k$. This is because we can always double the period, i.e. take $(k_1, \ldots, k_m, k_1, \ldots, k_m) \in \set{\pm 1}^{2m}$ instead of $(k_1, \ldots, k_m) \in \set{\pm 1}^m$ as mentioned in the introduction.

Lemma \ref{lemma1} has the following important corollary.

\begin{cor} \label{cor1}
Let $m \in \N$ and $k \in \set{\pm 1}^m$. Then $\spec(A^k_{per}) = p_k^{-1}([-2,2])$ if $k$ is even and $\spec(A^k_{per}) = p_k^{-1}(i[-2,2])$ if $k$ is odd.
\end{cor}

\begin{proof}
Let $k$ be even first. By Proposition \ref{thmD} and Lemma \ref{lemma1} we have
\begin{align*}
\spec(A^k_{per}) &= \set{\lambda \in \C : \det(a(\phi) - \lambda I_m) = 0 \text{ for some } \phi \in [0,2\pi)}\\
&= \set{\lambda \in \C : p_k(\lambda) - 2\cos(\phi) = 0 \text{ for some } \phi \in [0,2\pi)}\\
&= \set{\lambda \in \C : p_k(\lambda) \in [-2,2]}\\
&= p_k^{-1}([-2,2]).
\end{align*}
If $k$ is odd, just replace $-2\cos(\phi)$ by $2i\sin(\phi)$.
\end{proof}

The next proposition not only shows that $\spec(p_k(A^k_{per})) = [-2,2]$ ($i[-2,2]$) for all even (odd) $k \in \set{\pm 1}^m$ but also that $p_k(A^k_{per})$ has a very simple form.

\begin{prop} \label{prop1}
Let $m \in \N$, $k \in \set{\pm 1}^m$ and let $p_k$ be the corresponding polynomial given by Lemma \ref{lemma1}. If $k$ is even, then
\[p_k(a^k(\phi)) = 2\cos(\phi)I_m\]
and $p_k(A^k_{per})$ is the Laurent operator with $1$ on its $m$-th sub- and superdiagonal (and $0$ everywhere else):
\begin{equation} \label{eq4}
p_k(A^k_{per}) = \begin{pmatrix} & & \ddots & & & & \\ & & & 1 & & & \\ \ddots & & & & 1 & & \\ & 1 & & & & 1 & \\ & & 1 & & & & \ddots \\ & & & 1 & & & \\ & & & & \ddots & & \end{pmatrix}.
\end{equation}
If $k$ is odd, then
\[p_k(a^k(\phi)) = -2i\sin(\phi)I_m\]
and $p_k(A^k_{per})$ is the Laurent operator with $-1$ on its $m$-th subdiagonal and $1$ on its $m$-th superdiagonal (and $0$ everywhere else):
\begin{equation} \label{eq5}
p_k(A^k_{per}) = \begin{pmatrix} & & \ddots & & & & \\ & & & 1 & & & \\ \ddots & & & & 1 & & \\ & -1 & & & & 1 & \\ & & -1 & & & & \ddots \\ & & & -1 & & & \\ & & & & \ddots & & \end{pmatrix}.
\end{equation}
\end{prop}

\begin{proof}
In both cases the first part follows immediately from Lemma \ref{lemma1} and the theorem of Cayley-Hamilton. For the second part observe that $2\cos(\phi)I_m = (e^{i\phi} + e^{-i\phi})I_m$ is the symbol of \eqref{eq4} and $-2i\sin(\phi)I_m = (-e^{i\phi} + e^{-i\phi})I_m$ is the symbol of \eqref{eq5}. The assertion thus follows by Proposition \ref{thmD}.
\end{proof}

This simple observation now enables us to prove the first of our main results.

\begin{thm} \label{thm1}
Let $m \geq 2$, $k := (k_1, \ldots, k_{m-2}, -1, 1) \in \set{\pm 1}^m$ and $\hat{k} := (k_1, \ldots, k_{m-2}, 1, -1) \in \set{\pm 1}^m$. Furthermore, let $b \in \set{\pm 1}^{\Z}$ and
\[A^b := \begin{pmatrix} \ddots & \ddots & & & \\ \ddots & 0 & 1 & & \\ & b_0 & \fbox{$0$} & 1 & \\ & & b_1 & 0 & \ddots \\ & & & \ddots & \ddots \end{pmatrix},\]
where the box indicates $A^b_{0,0}$. If the corresponding polynomials $p_k$ and $p_{\hat{k}}$ are equal, then there exist $c \in \set{\pm 1}^{\Z}$ and $B \in \Lc(\ell^2(\Z))$ such that 
\[p_k(A^c) \cong B \oplus A^b,\]
where we consider the following decomposition of the Hilbert space $\ell^2(\Z)$:
\[\ell^2(\Z) \cong \ell^2(\Z \setminus m\Z) \oplus \ell^2(m\Z) \cong \ell^2(\Z) \oplus \ell^2(\Z).\]
In particular, $\spec(A^b) \subset \spec(p_k(A^c))$.
\end{thm}

\begin{proof}
Let $c \in \set{\pm 1}^{\Z}$ be the sequence defined by:
\begin{itemize}
	\item $c_0 = 1$, $c_1 = -1$,
	\item $c_{rm+j} = k_{j-1}$ for $j \in \set{2, \ldots, m-1}$, $r \in \Z$,
	\item $\prod\limits_{j = 0}^{m-1} c_{rm-j} = b_r$ for $r \in \Z$,
	\item $c_{rm+1} = -c_{rm}$ for $r \in \Z$.
\end{itemize}
Note that $A^c$ is very similar to $A^k_{per}$ and $A^{\hat{k}}_{per}$. The difference is that, depending on the sequence $b$, the entries $c_{rm}$ and $c_{rm+1}$ are swapped for some $r \in \Z$. For $m = 2$ this is exactly the same construction as in \cite[Lemma 5]{ChaDa}. First we will prove the following claim by induction:

\medskip

\noindent \textbf{Claim 1:} $((A^c)^s)_{i,i-s+2j}$ only depends on the coefficients $c_{i-s+j+1}, \ldots, c_{i+j}$ for $s \in \N$, $i \in \Z$ and $j \in \set{0, \ldots, s-1}$. Furthermore, $((A^c)^s)_{i,i+s} = 1$ for all $s \in \N$, $i \in \Z$ and all other entries are $0$.

\medskip

\noindent For $s = 1$ we have $(A^c)_{i,i-1+2j} = c_{i+j}$ for $j = 0$, $i \in \Z$, and $(A^c)_{i,i+1} = 1$ for $i \in \Z$. All other entries are $0$. So assume that the claim holds for $s-1$. Then
\begin{align} \label{eq7}
((A^c)^s)_{i,i-s+2j} &= (A^c)_{i,i+1}((A^c)^{s-1})_{i+1,i-s+2j} + (A^c)_{i,i-1}((A^c)^{s-1})_{i-1,i-s+2j}\notag\\
&= ((A^c)^{s-1})_{i+1,i-s+2j} + c_i((A^c)^{s-1})_{i-1,i-s+2j}\\
&= f(c_{i-s+j+2}, \dots, c_{i+j}) + c_ig(c_{i-s+j+1}, \ldots, c_{i+j-1}),\notag
\end{align}
where $f$ and $g$ are some polynomials. Observe that $c_i$ is contained in $\set{c_{i-s+j+1}, \ldots, c_{i+j}}$ for all $i \in \Z$, $j \in \set{0, \ldots, s-1}$. Thus $((A^c)^s)_{i,i-s+2j}$ only depends on the coefficients $c_{i-s+j+1}, \ldots, c_{i+j}$ for all $i \in \Z$, $j \in \set{0, \ldots, s-1}$. Plugging $j = s$ into \eqref{eq7} yields $((A^c)^s)_{i,i+s} = 1$ for all $i \in \Z$ by the same induction argument. Plugging in $j \in \Z+\frac{1}{2}$ and $j \in \Z \setminus \set{0, \ldots, s}$ shows that all other entries are $0$. This finishes the proof of the claim.

\medskip

Using Claim 1 with $s = m$ and $i = rm$ for $r \in \Z$, we get that $((A^c)^m)_{rm,(r-1)m+2j}$ only depends on the coefficients $c_{(r-1)m+j+1}, \ldots, c_{rm+j}$ for all $j \in \set{0, \ldots, m-1}$, $((A^c)^m)_{rm,rm+m} = 1$ and all other entries in row $rm$ are $0$. Similarly, $((A^c)^m)_{(r-1)m+2j,rm}$ only depends on the coeffcients $c_{(r-1)m+j+1}, \ldots, c_{rm+j}$ for all $j \in \set{0, \ldots, m-1}$, $((A^c)^m)_{rm+m,rm} = 1$ and all other entries in column $rm$ are $0$. Moreover, Claim 1 also implies that the same is true for $p_k(A^c)$ because $p_k$ is an even/odd monic polynomial of degree $m$ by Lemma \ref{lemma1}.

\medskip

\noindent \textbf{Claim 2:} $((A^c)^m)_{i,i-m} = \prod\limits_{j = 0}^{m-1} c_{i-j}$ for all $i \in \Z$.

\medskip

\noindent This again follows easily by induction:
\begin{align*}
((A^c)^m)_{i,i-m} &= (A^c)_{i,i-1}((A^c)^{m-1})_{i-1,i-m} + (A^c)_{i,i+1}((A^c)^{m-1})_{i+1,i-m}\\
&= c_i \cdot ((A^c)^{m-1})_{i-1,i-m} + 1 \cdot 0\\
&= c_i \cdot \ldots \cdot c_{i-(m-1)}.
\end{align*}
Since $((A^c)^s)_{i,i-m} = 0$ for all $s < m$, also $(p_k(A^c))_{i,i-m} = \prod\limits_{j = 0}^{m-1} c_{i-j}$ for all $i \in \Z$. By definition of $c$, it thus follows
\[(p_k(A^c))_{rm,(r-1)m} = \prod\limits_{j = 0}^{m-1} c_{rm-j} = b_r\]
for all $r \in \Z$. Furthermore, $(p_k(A^c))_{rm,(r-1)m+2j}$ only depends on the coefficients $c_{(r-1)m+j+1}, \ldots$, $c_{rm+j}$ for $j \in \set{1, \ldots, m-1}$. In particular, these numbers all depend on $c_{rm}$ and $c_{rm+1}$ but not on $c_{(r-1)m}$, $c_{(r-1)m+1}$, $c_{(r+1)m}$ or $c_{(r+1)m+1}$. This implies
\[(p_k(A^c))_{rm,(r-1)m+2j} = (p_k(A^k_{per}))_{rm,(r-1)m+2j} = 0 \quad (\text{if } c_{rm} = -1, c_{rm+1} = 1)\]
or (using $p_k = p_{\hat{k}}$)
\[(p_k(A^c))_{rm,(r-1)m+2j} = (p_{\hat{k}}(A^{\hat{k}}_{per}))_{rm,(r-1)m+2j} = 0 \quad (\text{if } c_{rm} = 1, c_{rm+1} = -1)\]
for $j \in \set{1, \ldots, m-1}$ by Proposition \ref{prop1}. In other words, the entries $(p_k(A^c))_{rm,(r-1)m+2j}$ ($j \in \set{1, \ldots, m-1}$) can not ``know'' whether we swapped some of the entries $c_{lm}$ and $c_{lm+1}$ ($l \in \Z$) or not. Thus they have to remain zero. Similarly, the entries $(p_k(A^c))_{(r-1)m+2j,rm}$ ($j \in \set{1, \ldots, m-1}$) remain $0$. Therefore $p_k(A^c)$ looks like this (where $*$ means ``some unimportant entries''):
\[p_k(A^c) = \left(\begin{array}{cc>{\columncolor[rgb]{.75,.75,.75}[4pt]}cccc>{\columncolor[rgb]{.75,.75,.75}[0pt]}cccc>{\columncolor[rgb]{.75,.75,.75}[6pt]}ccc} & \vdots & \vdots & \vdots & \ddots & & & & & & & & \\ \hdots & * & \vdots & * & \hdots & * & & & & & & & \\ \rowcolor[rgb]{.75,.75,.75}[5pt][5pt] \hdots & \hdots & 0 & \hdots & \hdots & 0 & 1 & & & & & & \\ \hdots & * & \vdots & * & \hdots & * & 0 & * & & & & & \\ \ddots & \vdots & \vdots & \vdots & & \vdots & \vdots & \vdots & \ddots & & & & \\ & * & 0 & * & \hdots & * & \vdots & * & \hdots & * & & & \\ \rowcolor[rgb]{.75,.75,.75}[5pt][5pt] & & b_r & 0 & \hdots & \hdots & 0 & \hdots & \hdots & 0 & 1 & & \\ & & & * & \hdots & * & \vdots & * & \hdots & * & 0 & * & \\ & & & & \ddots & \vdots & \vdots & \vdots & & \vdots & \vdots & \vdots & \ddots \\ & & & & & * & 0 & * & \hdots & * & \vdots & * & \hdots \\ \rowcolor[rgb]{.75,.75,.75}[5pt][5pt] & & & & & & b_{r+1} & 0 & \hdots & \hdots & 0 & \hdots & \hdots \\ & & & \multispan{3}\mathstrut\upbracefill \, & & * & \hdots & * & \vdots & * & \hdots \\ & & & \multicolumn{3}{c}{m-1 \text{ columns}} & & & \ddots & \vdots & \vdots & \vdots & \end{array}\right).\]

Decomposing our Hilbert space $\ell^2(\Z) \cong \ell^2(\Z \setminus m\Z) \oplus \ell^2(m\Z) \cong \ell^2(\Z) \oplus \ell^2(\Z)$, we get the following decomposition of $p_k(A^c)$:
\[p_k(A^c) \cong B \oplus A^b\]
for some $B \in \Lc(\ell^2(\Z))$. In particular, $\spec(A^b) \subset \spec(p_k(A^c))$.
\end{proof}

By construction of the sequence $c$, we also have the following important corollary for periodic operators.

\begin{cor} \label{cor2}
Under the same assumptions as in Theorem \ref{thm1}, we have that if $b \in \set{\pm 1}^{\Z}$ is an $n$-periodic sequence with even period $(b_1, \ldots, b_n)$, then $c$, as defined in the proof of Theorem \ref{thm1}, is an $nm$-periodic sequence with even period $(c_1, \ldots, c_{nm})$. $B \in \Lc(\ell^2(\Z))$ then is a periodic operator, too. Furthermore, if we denote the symbols of $A^b$, $A^c$ and $B$ by $a^b$, $a^c$ and $a^B$, then also $p_k(a^c)$ can be decomposed as $p_k(a^c) \cong a^B \oplus a^b$. In particular, $\spec(a^b(\phi)) \subset \spec(p_k(a^c(\phi)))$ for every $\phi \in [0,2\pi)$.
\end{cor}

\begin{proof}
The first part follows by contruction of $c$. It remains to prove that the symbol $p_k(a^c)$ can be decomposed in a similar way. By Proposition \ref{thmD}, $A^c$ is unitarily equivalent to the multiplication operator $M_{a^c} \in \Lc(L^2([0,2\pi),\C^{nm}))$. Let us denote this equivalence by $\Fc_{nm}$, i.e. $A^c = \Fc_{nm}^*M_{a^c}\Fc_{nm}$. It follows 
\[p_k(A^c) = p_k(\Fc_{nm}^*M_{a^c}\Fc_{nm}) = \Fc_{nm}^*p_k(M_{a^c})\Fc_{nm} = \Fc_{nm}^*M_{p_k(a^c)}\Fc_{nm}.\]
Furthermore, $A^b$ is unitarily equivalent to the multiplication operator $M_{a^b} \in \Lc(L^2([0,2\pi),\C^{m}))$ and $B$ is unitarily equivalent to the multiplication operator $M_{a^B} \in \Lc(L^2([0,2\pi),\C^{(n-1)m}))$. Let us denote these equivalences by $\Fc_m$ and $\Fc_{(n-1)m}$. Furthermore, let us denote the decomposition
$\ell^2(\Z) \cong \ell^2(\Z \setminus m\Z) \oplus \ell^2(m\Z)$ by $U$ and the decomposition $L^2([0,2\pi),\C^{nm}) \cong L^2([0,2\pi),\C^{(n-1)m}) \oplus L^2([0,2\pi),\C^m)$ (in the obvious way) by $V$. It is not hard to see that
\[(\Fc_{(n-1)m} \oplus \Fc_m)U\Fc_{nm}^* = V\]
holds. Thus
\begin{align*}
M_{p_k(a^c)} &= \Fc_{nm}p_k(A^c)\Fc_{nm}^*\\
&= \Fc_{nm}U^*(B \oplus A^b)U\Fc_{nm}^*\\
&= \Fc_{nm}U^*(\Fc_{(n-1)m}^*M_{a^B}\Fc_{(n-1)m} \oplus \Fc_m^*M_{a^b}\Fc_m )U\Fc_{nm}^*\\
&= V^*(M_{a^B} \oplus M_{a^b})V.
\end{align*}
This implies $p_k(a^c) \cong a^B \oplus a^b$ (as functions of $\phi$).
\end{proof}

\begin{defn} \label{defn2}
We define
\[S := \set{p_k : p_k \text{ is a polynomial such that the assumptions of Theorem \ref{thm1} are satisfied}}\]
as our set of symmetries of $\pi_\infty$.
\end{defn}

In the case of periodic operators we can prove the following stronger version of Theorem \ref{thm1} that justifies the definition of $S$.

\begin{figure}[ht!]
\begin{center}
\includegraphics[width = \textwidth, trim = 15mm 20mm 10mm 10mm, clip = true]{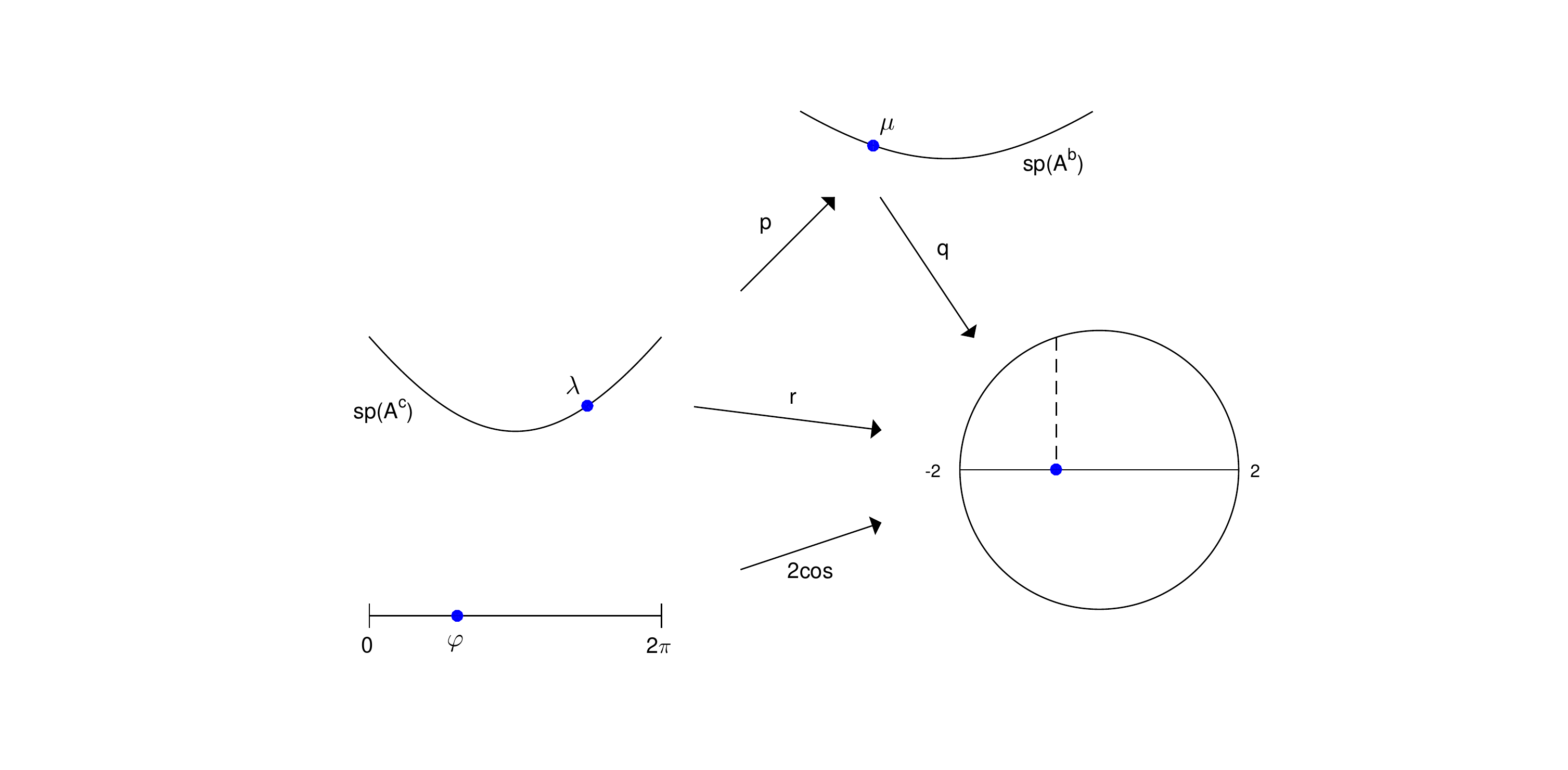}
\end{center}

\caption{Schematic picture of the maps involved in the proof of Theorem \ref{thm2}.}
\end{figure}

\begin{thm} \label{thm2}
Let $n \in \N$, let $b \in \set{\pm 1}^{\Z}$ be an $n$-periodic sequence and let $k \in \set{\pm 1}^m$ be such that $p := p_k \in S$. Moreover, let $c \in \set{\pm 1}^{\Z}$ be the sequence as constructed in the proof of Theorem \ref{thm1}. Then the following assertion holds:
\[p(\lambda) \in \spec(A^b) \Longleftrightarrow \lambda \in \spec(A^c).\]
\end{thm}

\begin{proof}
W.l.o.g. we can assume that $(b_1, \ldots, b_n)$ is even. By Corollary \ref{cor2}, $(c_1, \ldots, c_{nm})$ is even, too. Let $q$ and $r$ be the polynomials given by Lemma \ref{lemma1} corresponding to $b$ and $c$. Also denote by $a^b$ and $a^c$ the symbols of $A^b$ and $A^c$. Fix some $\phi \in [0,2\pi)$ and let $\mu \in \C$ be an eigenvalue of $a^b(\phi)$. Again by Corollary \ref{cor2}, there exists some $\lambda \in \spec(a^c(\phi))$ such that $p(\lambda) = \mu$. Since $q(\mu) = 2\cos(\phi)$ by Proposition \ref{prop1}, we have $q(p(\lambda)) = 2\cos(\phi)$. On the other hand, also $r(\lambda) = 2\cos(\phi)$ by Proposition \ref{prop1}. Thus
\[(q \circ p)(\lambda) = 2\cos(\phi) = r(\lambda)\]
(cf. Figure 4).
Since both $q \circ p$ and $r$ are polynomials and the above argument is valid for every $\phi \in [0,2\pi)$, we conclude that $q \circ p$ and $r$ are equal.

It follows
\[\spec(A^c) = r^{-1}([-2,2]) = p^{-1}(q^{-1}([-2,2])) = p^{-1}(\spec(A^b))\]
by Corollary \ref{cor1}.
\end{proof}

Theorem \ref{thm2} combined with the result $\D \subset \clos(\pi_\infty)$ from \cite{ChaDa} implies the following corollary.

\begin{cor} \label{cor3}
Let $p \in S$. Then $p^{-1}(\pi_\infty) \subset \pi_\infty$ and hence $p^{-1}(\D) \subset p^{-1}(\clos(\pi_\infty)) \subset \clos(\pi_\infty)$.
\end{cor}

Recall that $k = (k_1, \ldots, k_m) \in \set{\pm 1}^m$ generates a polynomial $p = p_k \in S$ if $k_{m-1} \neq k_m$, so that $\hat{k} = (k_1, \ldots, k_{m-2},k_m,k_{m-1}) \neq k$ but still $p_k = p_{\hat{k}}$. Thus it is immediate that all $k$ of the form $(1, \ldots, 1, -1, 1)$ and $(-1, \ldots, -1, -1, 1)$ generate a polynomial $p_k \in S$. Indeed, if $k = (1, \ldots, 1, -1, 1)$ and $\hat{k} = (1, \ldots, 1, 1, -1)$, then $A^k_{per}$ and $A^{\hat{k}}_{per}$ are unitarily equivalent by a simple shift. This implies that $S$ contains a countable number of polynomials. However, there are also a lot more than these trivial examples as the following table shows. We conjecture that there are approximately $2^{\left\lceil\frac{m}{2}\right\rceil-1}$ polynomials of degree $m$ in $S$.

\renewcommand{\arraystretch}{1.4}

\begin{center}
\Large{
\begin{tabular}{|c|c|c|} \hline
 No. & $k$                     & $p_k(\lambda)$                                  \\ \hline
 2.1 & $(-1,1)$                & $\lambda^2$                                     \\ \hline
 3.1 & $(1,-1,1)$              & $\lambda^3 - \lambda$                           \\ \hline
 3.2 & $(-1,-1,1)$             & $\lambda^3 + \lambda$                           \\ \hline
 4.1 & $(1,1,-1,1)$            & $\lambda^4 - 2\lambda^2$                        \\ \hline
 4.2 & $(-1,-1,-1,1)$          & $\lambda^4 + 2\lambda^2$                        \\ \hline
 5.1 & $(1,1,1,-1,1)$          & $\lambda^5 - 3\lambda^3 + \lambda$              \\ \hline
 5.2 & $(1,-1,1,-1,1)$         & $\lambda^5 - \lambda^3 + \lambda$               \\ \hline
 5.3 & $(-1,1,-1,-1,1)$        & $\lambda^5 + \lambda^3 + \lambda$               \\ \hline
 5.4 & $(-1,-1,-1,-1,1)$       & $\lambda^5 + 3\lambda^3 + \lambda$              \\ \hline
 6.1 & $(1,1,1,1,-1,1)$        & $\lambda^6 - 4\lambda^4 + 3\lambda^2$           \\ \hline
 6.2 & $(1,-1,-1,1,-1,1)$      & $\lambda^6 - \lambda^2$                         \\ \hline
 6.3 & $(-1,-1,-1,-1,-1,1)$    & $\lambda^6 + 4\lambda^4 + 3\lambda^2$           \\ \hline
 7.1 & $(1,1,1,1,1,-1,1)$      & $\lambda^7 - 5\lambda^5 + 6\lambda^3 - \lambda$ \\ \hline
 7.2 & $(1,1,-1,1,1,-1,1)$     & $\lambda^7 - 3\lambda^5 + 2\lambda^3 + \lambda$ \\ \hline
 7.3 & $(1,-1,1,-1,1,-1,1)$    & $\lambda^7 - \lambda^5 + 2\lambda^3 - \lambda$  \\ \hline
 7.4 & $(1,-1,-1,-1,1,-1,1)$   & $\lambda^7 + \lambda^5 - 2\lambda^3 + \lambda$  \\ \hline
 7.5 & $(-1,1,1,1,-1,-1,1)$    & $\lambda^7 - \lambda^5 - 2\lambda^3 - \lambda$  \\ \hline
 7.6 & $(-1,1,-1,1,-1,-1,1)$   & $\lambda^7 + \lambda^5 + 2\lambda^3 + \lambda$  \\ \hline
 7.7 & $(-1,-1,1,-1,-1,-1,1)$  & $\lambda^7 + 3\lambda^5 + 2\lambda^3 - \lambda$ \\ \hline
 7.8 & $(-1,-1,-1,-1,-1,-1,1)$ & $\lambda^7 + 5\lambda^5 + 6\lambda^3 + \lambda$ \\ \hline
\end{tabular}
}
\captionof{table}{Short list of elements in $S$.}
\end{center}

\begin{rem} \label{remark}
As mentioned in the introduction, we can iterate Corollary \ref{cor3} to get
\[U := \bigcup\limits_{n \in \N} p^{-n}(\D) \subset \bigcup\limits_{n \in \N} p^{-n}(\clos(\pi_\infty)) \subset \clos(\pi_\infty)\]
for every $p \in S$. In other words, $z \in U$ if and only if $\abs{p^n(z)} \leq 1$ for some $n \in \N$. Thus there is clearly a connection to the filled Julia set corresponding to $p$ which is given by
\[J_f(p) := \set{z \in \C : (p^n(z))_{n \in \N} \text{ is bounded}}.\]
(see \cite[Lemma 17.1]{Milnor}). Indeed, the boundary $J(p) := \partial J_f(p)$ (which is usually just called the Julia set corresponding to $p$) is contained in the closure of $\bigcup\limits_{n \in \N} p^{-n}(z)$ for every $z \in \C$ except for at most one point (see \cite[Corollary 14.8(a)]{Falconer}). Hence $J(p) \subset \clos(U) \subset \clos(\pi_\infty)$. Considering the pictures (viii) and (ix) below, it seems natural to conjecture that even the filled Julia set $J_f(p)$ is contained in $\clos(U)$. 
\end{rem}

We conclude with some pictures of subsets of $\Sigma$. The red unit circle serves as a reference.


\renewcommand*{\thesubfigure}{\roman{subfigure}}

\begin{figure}[ht!]
\begin{center}
\begin{subfigure}{0.4\textwidth}
\includegraphics[width = \textwidth,trim = 15mm 5mm 10mm 5mm, clip = true]{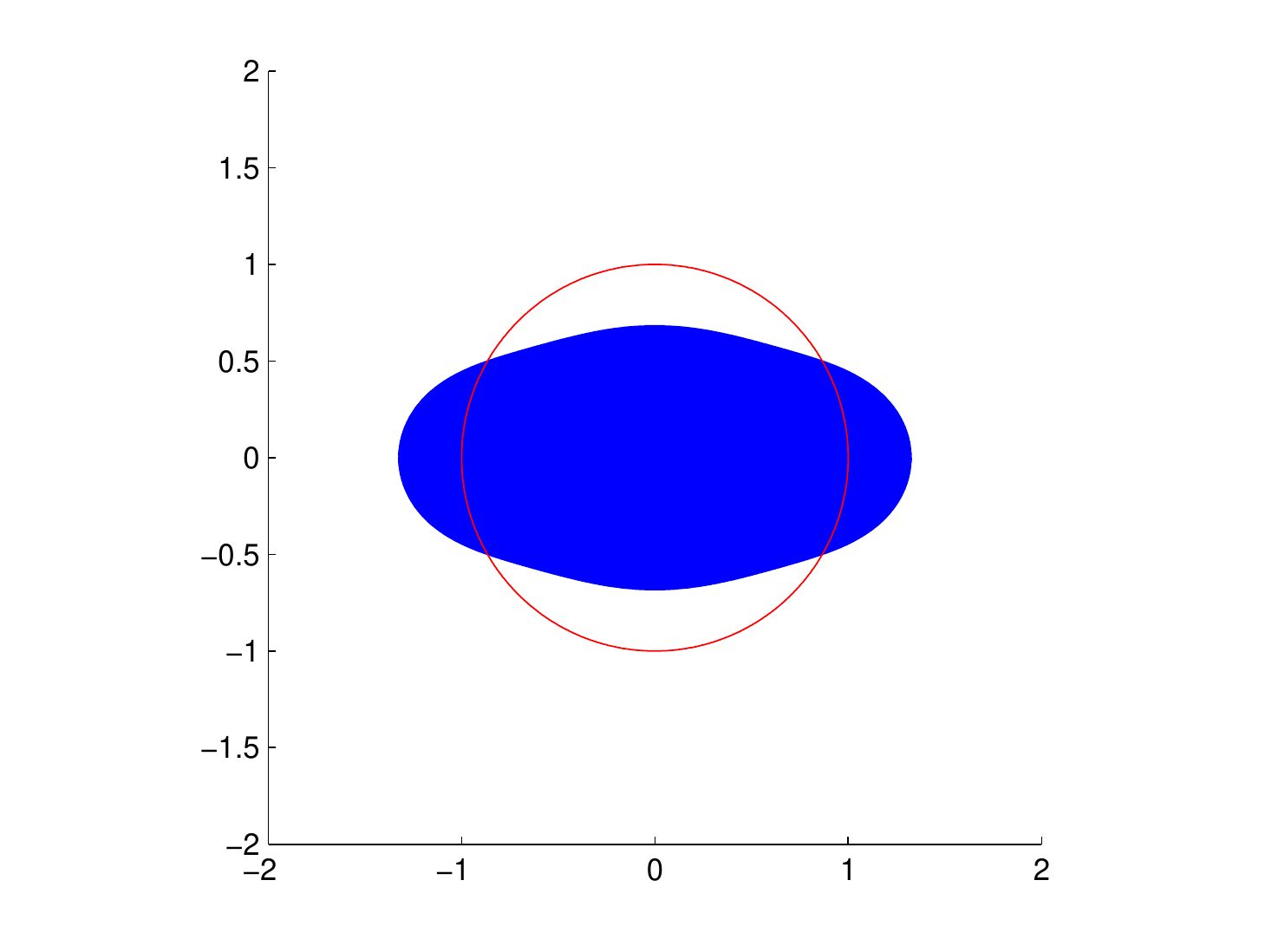}
\caption{\footnotesize $p_k^{-1}(\D)$ for $k = (1,-1,1)$}
\end{subfigure}
\begin{subfigure}{0.4\textwidth}
\includegraphics[width = \textwidth, trim = 15mm 5mm 10mm 5mm, clip = true]{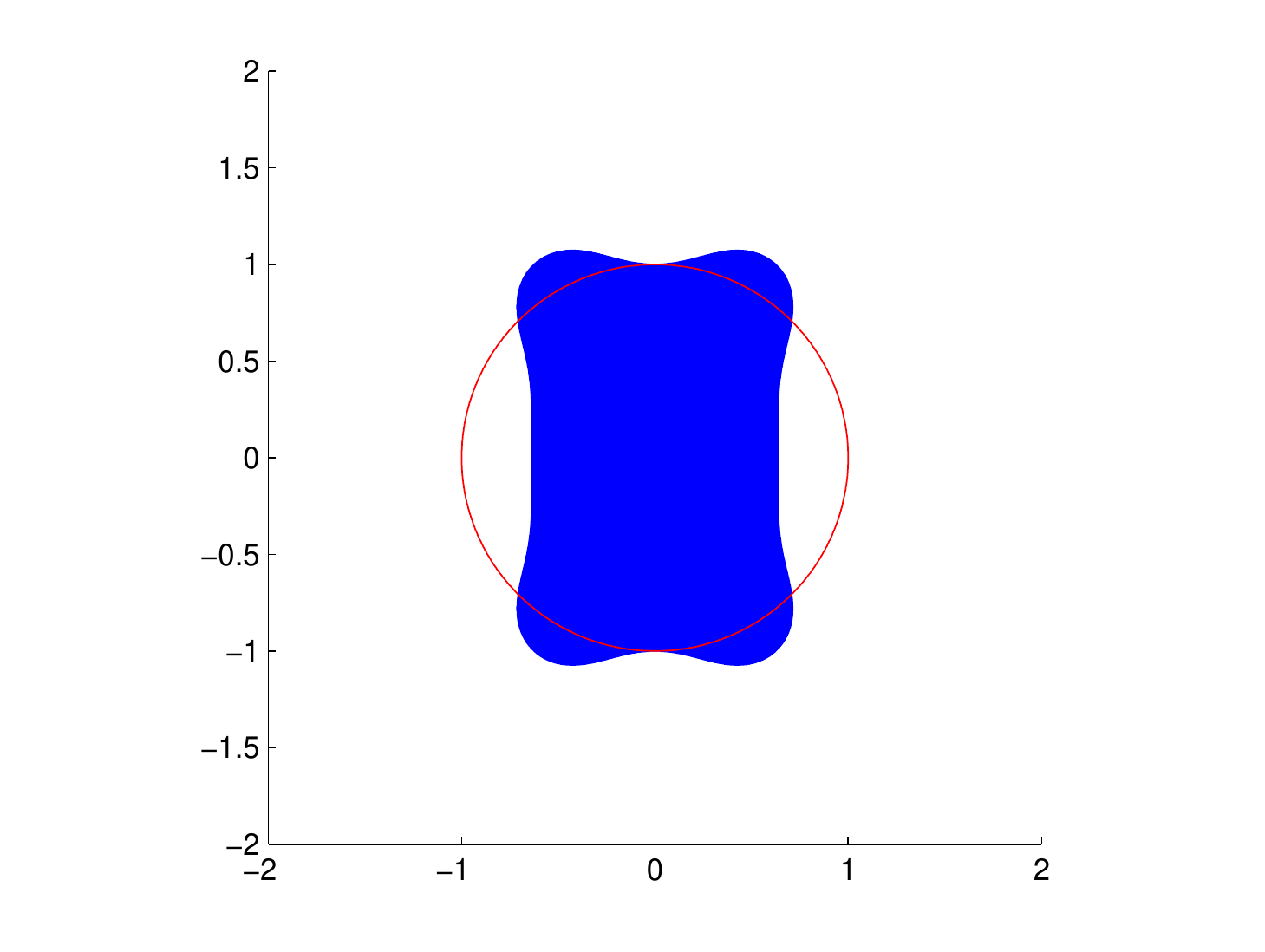}
\caption{\footnotesize $p_k^{-1}(\D)$ for $k = (-1,1,-1,-1,1)$}
\end{subfigure}
\begin{subfigure}{0.4\textwidth}
\includegraphics[width = \textwidth, trim = 15mm 5mm 10mm 5mm, clip = true]{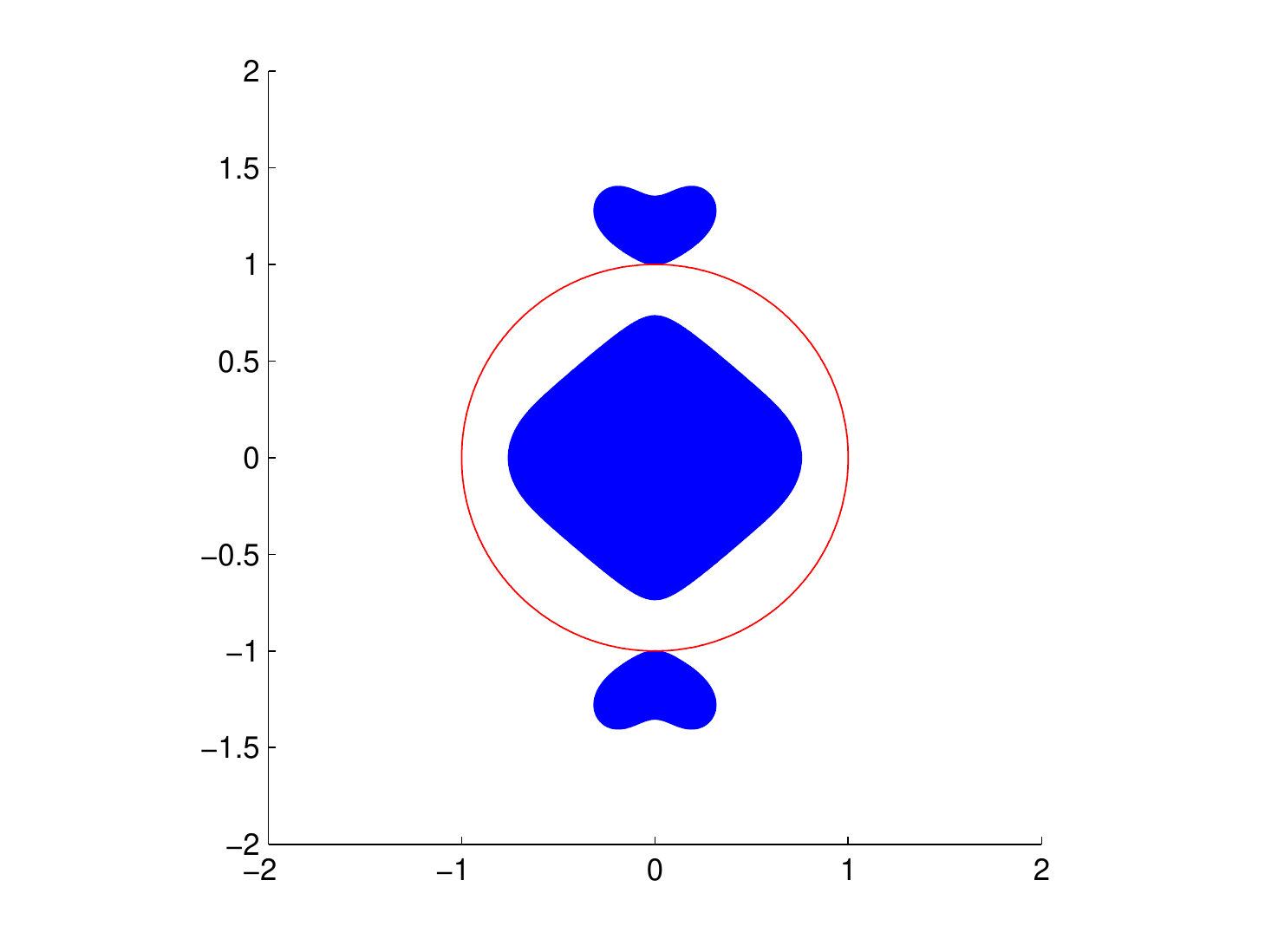}
\caption{\footnotesize $p_k^{-1}(\D)$ for $k = (-1,-1,1,-1,-1,-1,1)$}
\end{subfigure}
\begin{subfigure}{0.4\textwidth}
\includegraphics[width = \textwidth, trim = 15mm 5mm 10mm 5mm, clip = true]{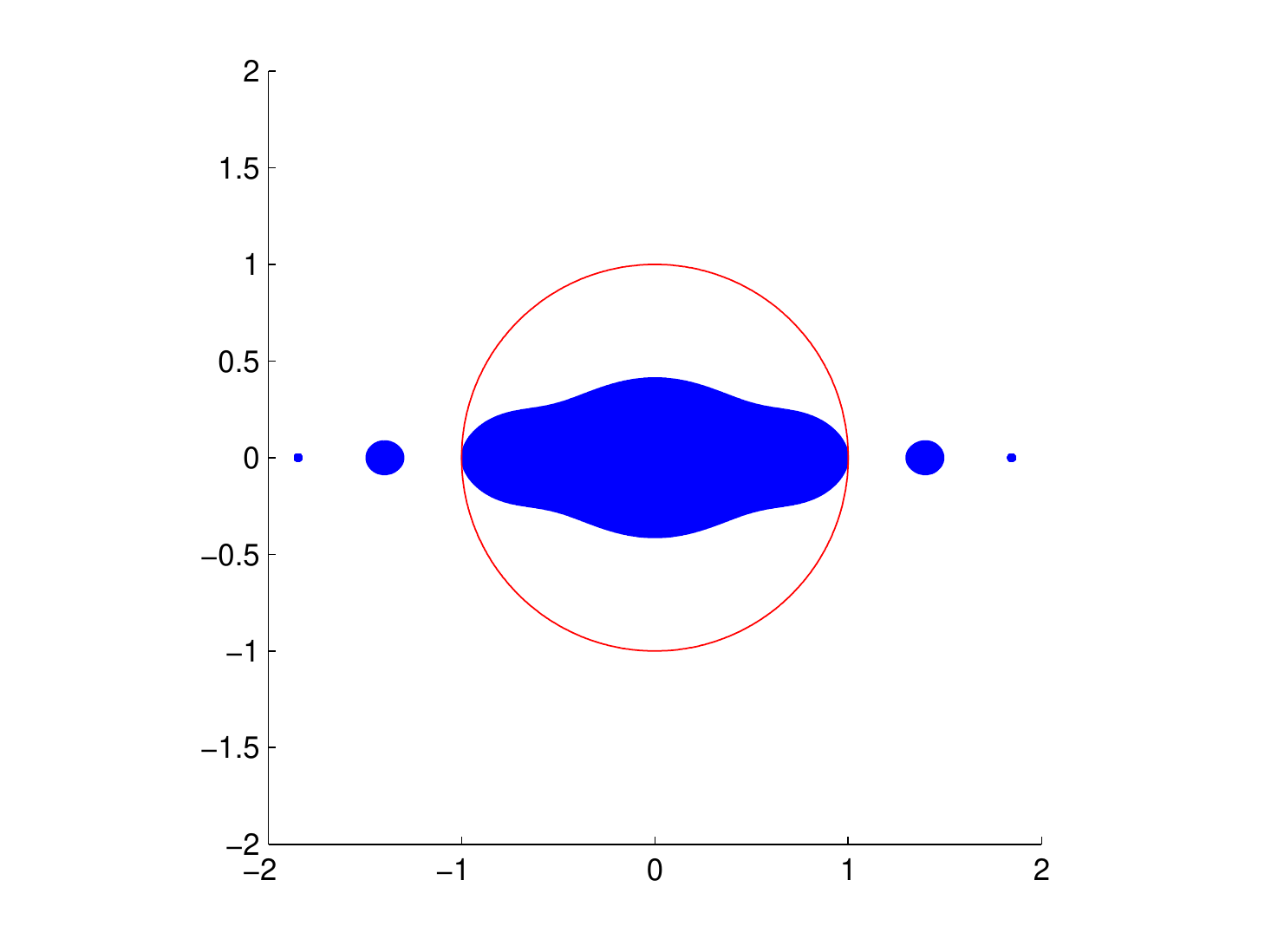}
\caption{\footnotesize $p_k^{-1}(\D)$ for $k = (1,1,1,1,1,1,-1,1)$}
\end{subfigure}
\end{center}
\end{figure}

\begin{figure}[ht!]
\begin{center}
\begin{subfigure}{0.4\textwidth}
\includegraphics[width = \textwidth, trim = 15mm 5mm 10mm 5mm, clip = true]{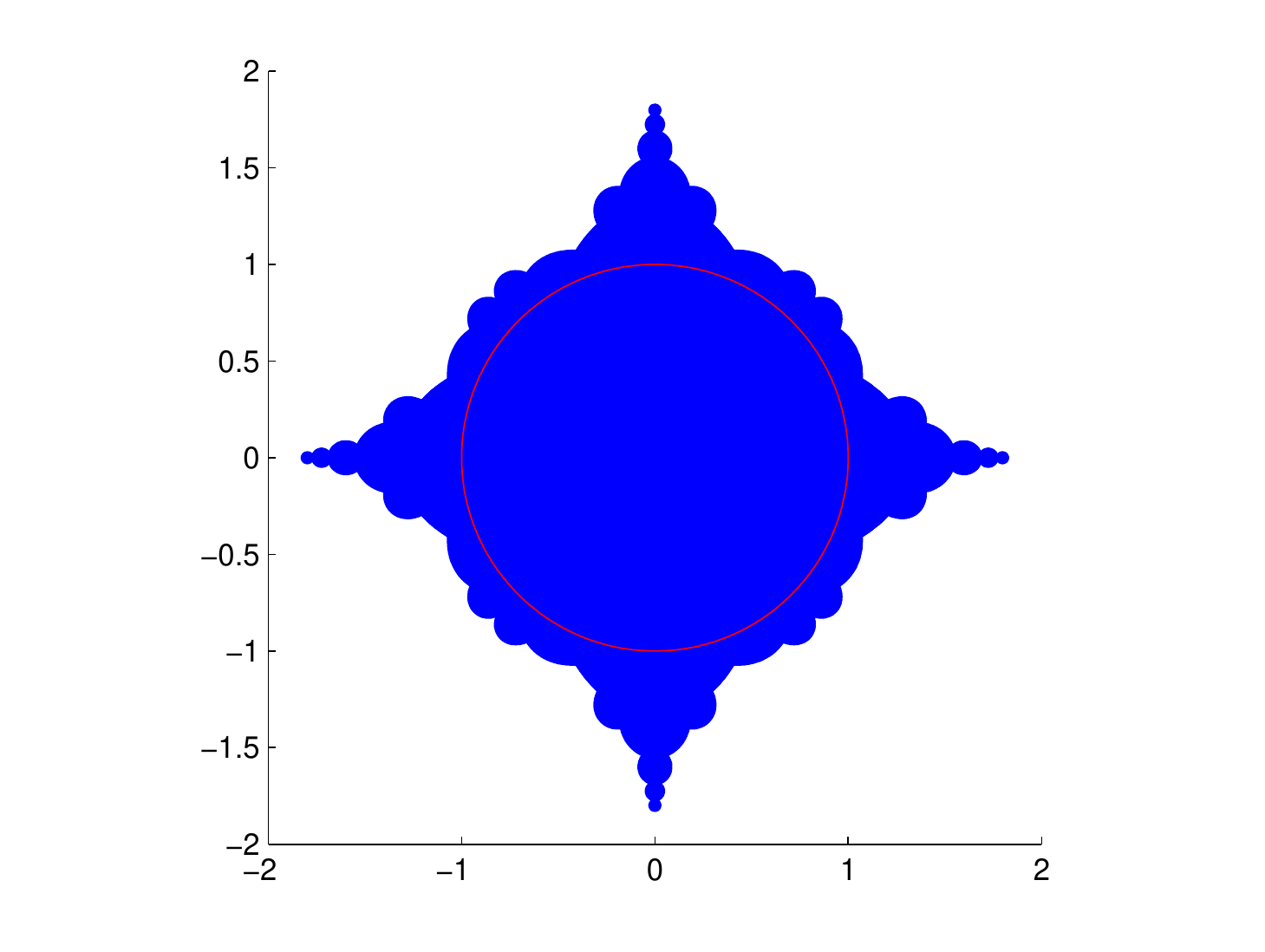}
\caption{\footnotesize $\bigcup\limits_{\substack{p \in S \\ \deg(p) \leq 7}} p^{-1}(\D)$}
\end{subfigure}
\begin{subfigure}{0.4\textwidth}
\includegraphics[width = \textwidth, trim = 15mm 7mm 10mm 5mm, clip = true]{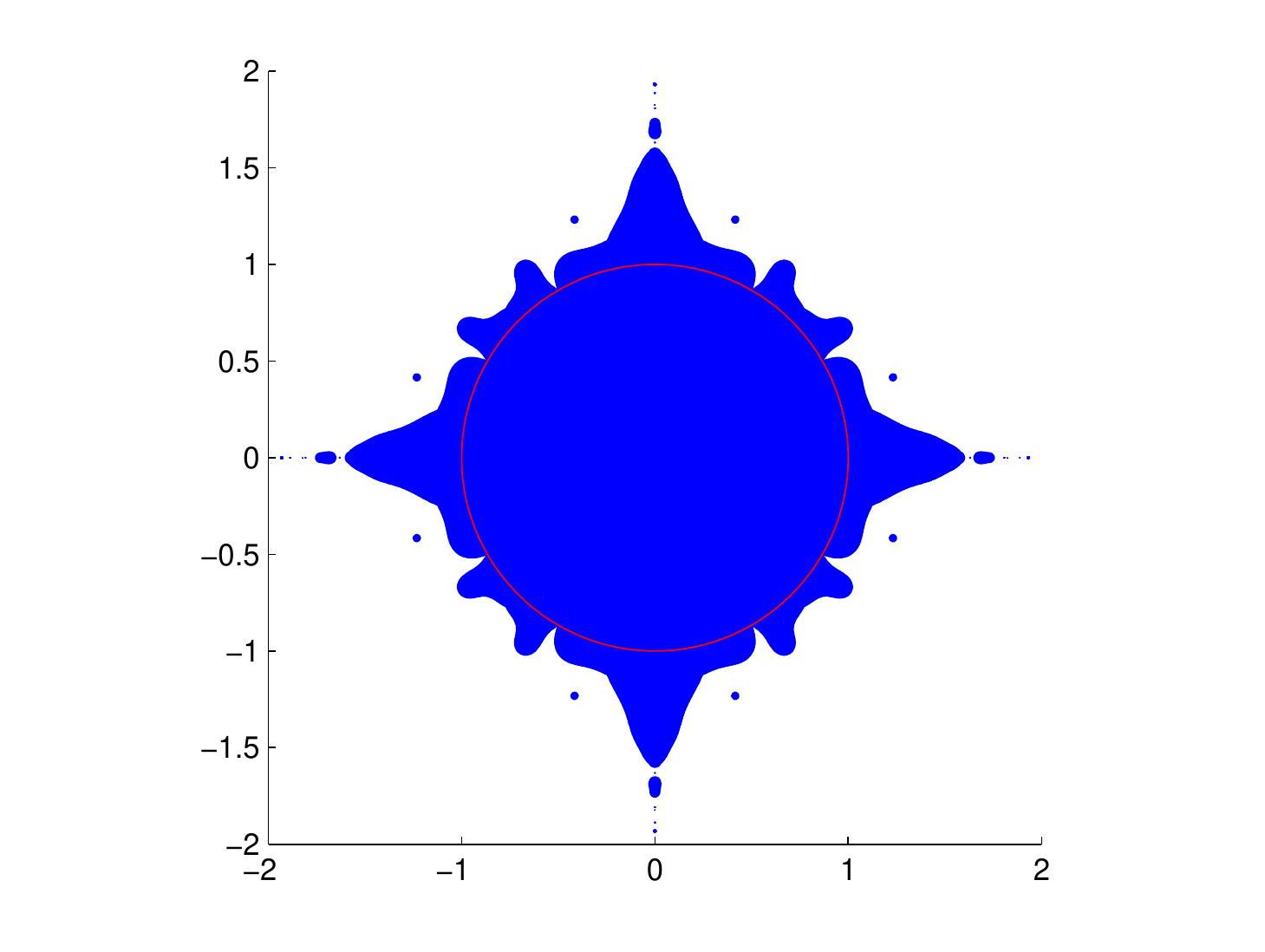}
\caption{\footnotesize $\bigcup\limits_{\substack{p \in S \\ \deg(p) = 12}} p^{-1}(\D)$}
\end{subfigure}
\end{center}
\begin{center}
\begin{subfigure}{0.9\textwidth}
\includegraphics[width = \textwidth, trim = 15mm 5mm 10mm 5mm, clip = true]{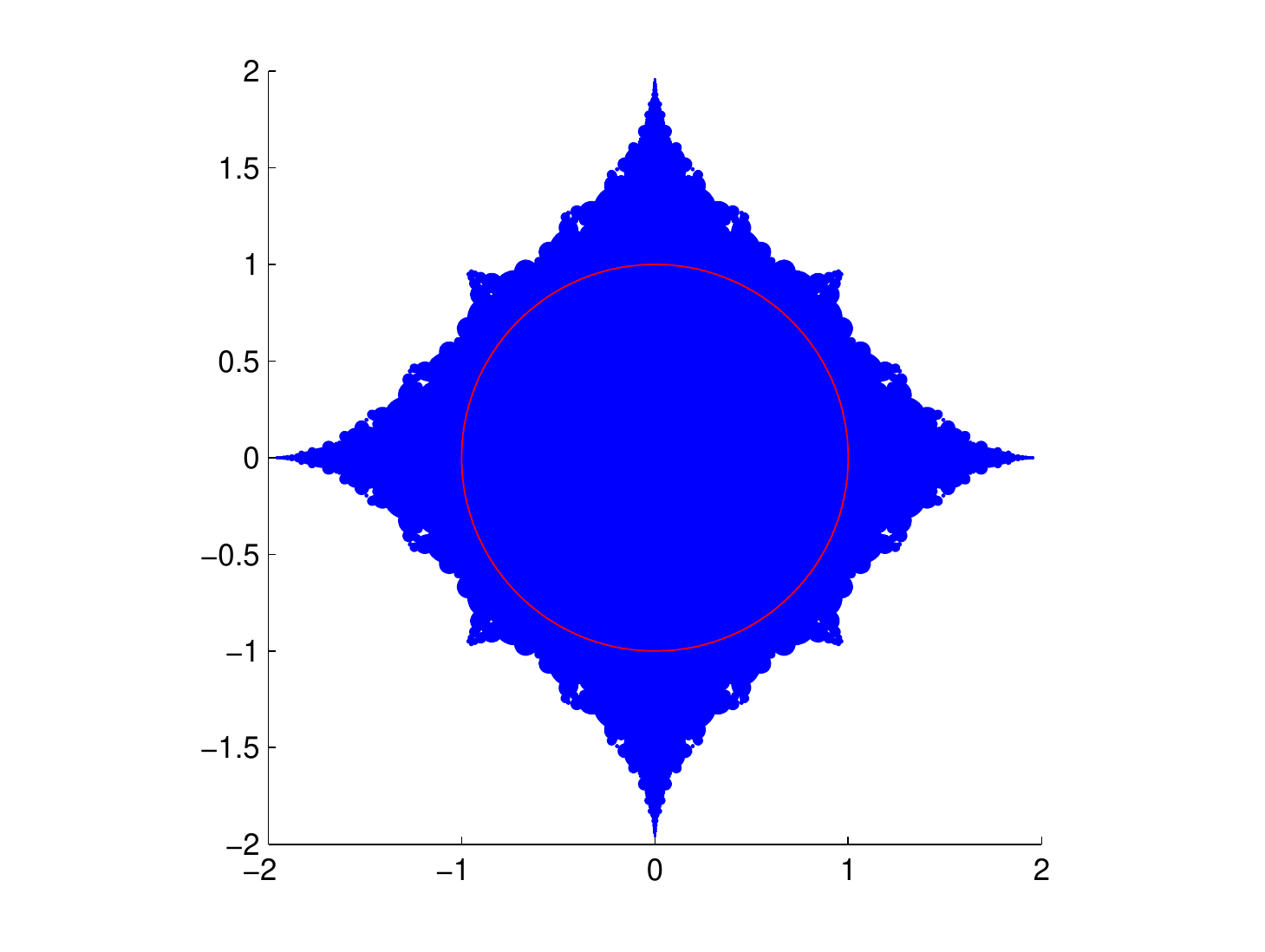}
\caption{\footnotesize $\bigcup\limits_{\substack{p \in S \\ \deg(p) \leq 15}} p^{-1}(\D)$}
\end{subfigure}
\end{center}
\end{figure}

\begin{figure}[ht!]
\begin{center}
\begin{subfigure}{0.4\textwidth}
\includegraphics[width = \textwidth,trim = 15mm 5mm 10mm 5mm, clip = true]{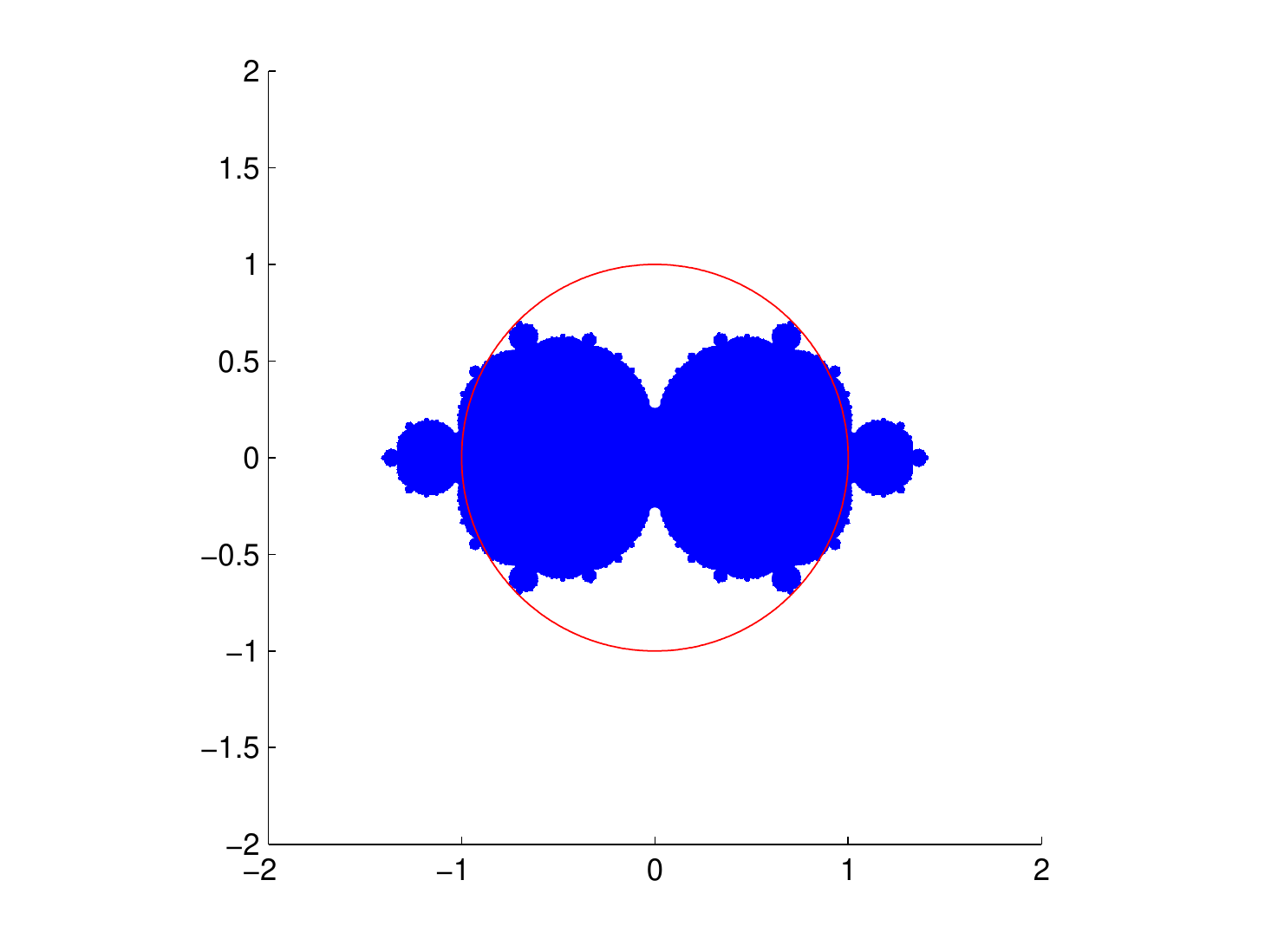}
\caption{\footnotesize $p_k^{-9}(\D)$ for $k = (1,-1,1)$}
\end{subfigure}
\begin{subfigure}{0.4\textwidth}
\includegraphics[width = \textwidth, trim = 15mm 5mm 10mm 5mm, clip = true]{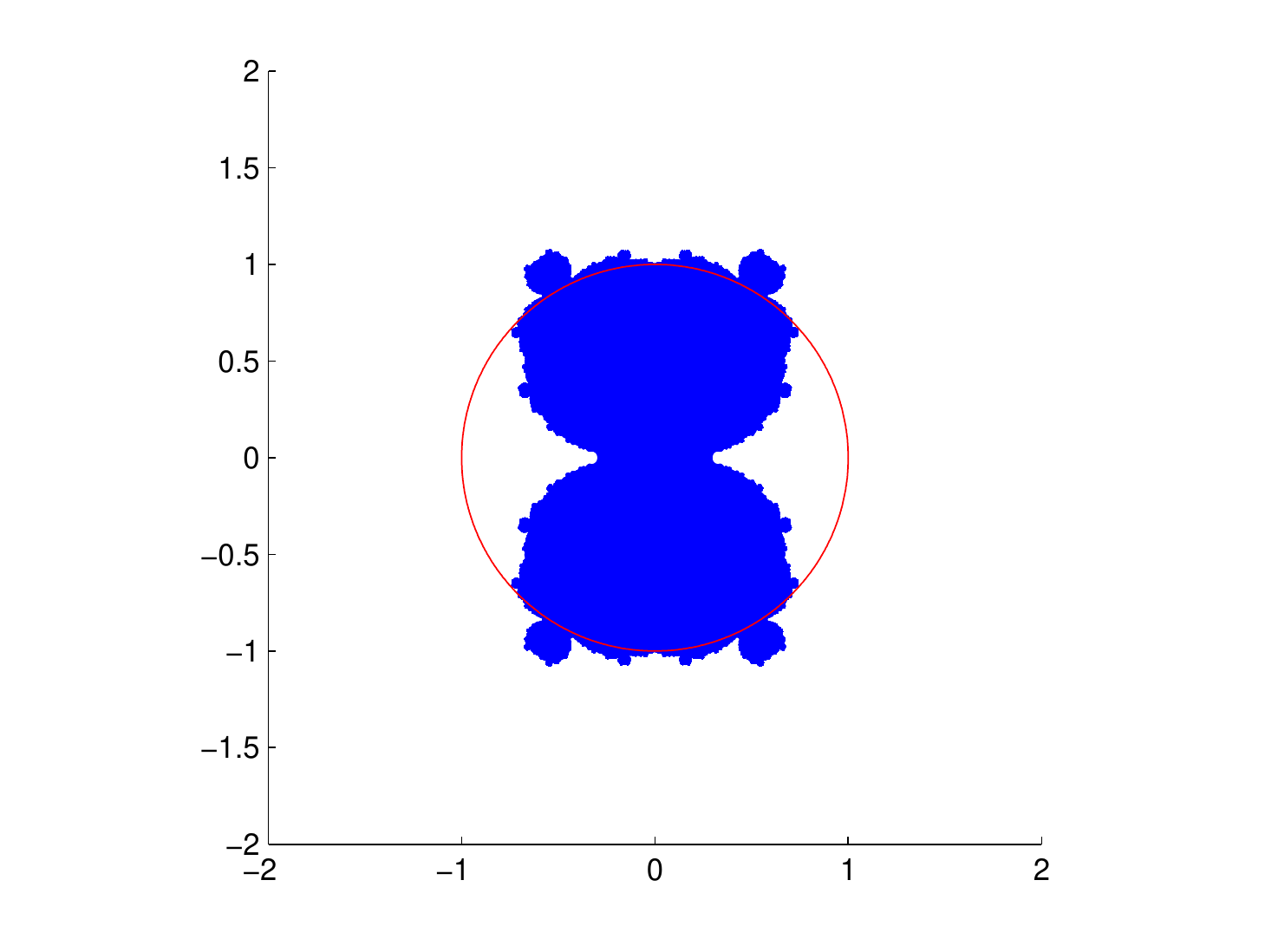}
\caption{\footnotesize $p_k^{-6}(\D)$ for $k = (-1,1,-1,-1,1)$}
\end{subfigure}
\begin{subfigure}{\textwidth}
\includegraphics[width = \textwidth, trim = 15mm 5mm 10mm 5mm, clip = true]{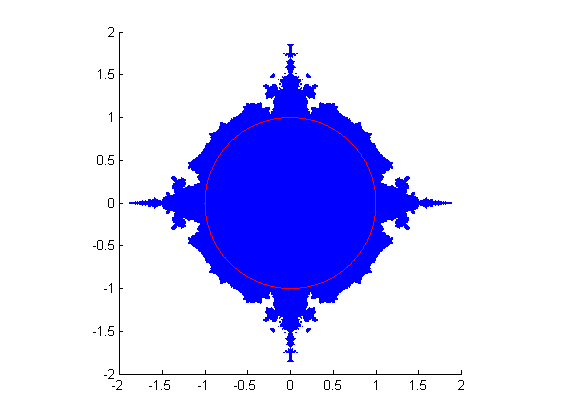}
\caption{\footnotesize $\bigcup\limits_{\substack{p \in S \\ \deg(p) \leq 8}} p^{-5}(\D)$}
\end{subfigure}
\end{center}
\end{figure}

\bigskip

\noindent
{\bf Author's address:}

\medskip

\noindent
Raffael Hagger \hfill{\tt raffael.hagger@tuhh.de}\\
Institute of Mathematics\\
Hamburg University of Technology\\
Schwarzenbergstr. 95 E\\
D-21073 Hamburg\\
GERMANY


\clearpage

\begin{thebibliography}{}
  \bibitem{ChaChoLi} S.N. Chandler-Wilde, R. Chonchaiya and M. Lindner: \emph{Eigenvalue Problem meets Sierpinski Triangle: Computing the Spectrum of a Non-Self-Adjoint Random Operator}, Operators and Matrices, 5 (2011), 633-648.
  \bibitem{ChaChoLi2} S.N. Chandler-Wilde, R. Chonchaiya and M. Lindner: \emph{On the Spectra and Pseudospectra of a Class of non-self-adjoint Random Matrices and Operators}, Operators and Matrices, 7 (2013), 739-775.
  \bibitem{ChaDa} S.N. Chandler-Wilde and E.B. Davies: \emph{Spectrum of a Feinberg-Zee Random Hopping Matrix}, Journal of Spectral Theory, 2 (2012), 147-179.
	\bibitem{Davies2} E.B. Davies: \emph{Linear Operators and their Spectra}, Cambridge Studies in Advanced Mathematics 106, Cambridge University Press, Cambridge, 2007.
	\bibitem{Falconer} K.~Falconer: \emph{Fractal Goemetry: Mathematical Foundations and Applications}, 2nd Ed., John Wiley, Chichester, 2003.
	\bibitem{FeZee} J. Feinberg and A. Zee: \emph{Spectral Curves of Non-Hermitean Hamiltonians}, Nucl. Phys. B 552 (1999), 599-623.
	\bibitem{Ha} R. Hagger: \emph{On the Spectrum and Numerical Range of Tridiagonal Random Operators}, Preprint at arXiv: 1407.5486.
	\bibitem{Ha2} R. Hagger: \emph{The Eigenvalues of Tridiagonal Sign Matrices are Dense in the Spectra of Periodic Tridiagonal Sign Operators}, Preprint at arXiv: 1412.1724
	\bibitem{HoOrZee} D.E. Holz, H. Orland and A. Zee: \emph{On the Remarkable Spectrum of a Non-Hermitian Random Matrix Model}, J. Phys. A Math. Gen 36 (2003), 3385-3400.
	\bibitem{Milnor} J. Milnor: \emph{Dynamics in one Complex Variable. Introductory Lectures}, Vieweg \& Sohn, Braunschweig, 1999.
\end{thebibliography}
\end{document}